\documentclass[english]{sobraep}

\title{The Positive Effects of Stochastic Rounding in Numerical Algorithms}

\author{El-Mehdi El Arar$^{1}$, Devan Sohier$^{1}$, Pablo de Oliveira Castro$^{1}$, Eric Petit$^{2}$\\
	\normalsize $^{1}$Universit\'e Paris-Saclay, UVSQ, LI-PaRAD\\
	\normalsize $^{2}$Intel Corp\\
	\normalsize  \{el-mehdi.el-arar, devan.sohier, pablo.oliveira\}@uvsq.fr\\
	\normalsize eric.petit@intel.com
}

\begin{document}

\maketitle

\begin{abstract}
Recently, stochastic rounding (SR) has been implemented in specialized hardware but most current computing nodes do not yet support this rounding mode. Several works empirically illustrate the benefit of stochastic rounding in various fields such as neural networks and ordinary differential equations. 
For some algorithms, such as summation, inner product or matrix-vector multiplication, it has been proved that SR provides probabilistic error bounds better than the traditional deterministic bounds.

In this paper, we extend this theoretical ground for a wider adoption of SR in computer architecture.
First, we analyze the biases of the two SR modes: SR-nearness and SR-up-or-down. We demonstrate on a case-study of Euler's forward method that IEEE-754 default rounding modes and SR-up-or-down accumulate rounding errors across iterations and that SR-nearness, being unbiased, does not. 
Second, we prove a $O(\sqrt{n})$ probabilistic bound on the forward error of Horner's polynomial evaluation method with SR, improving on the known deterministic $O(n)$ bound. 
\end{abstract}

\begin{keywords}
Stochastic rounding, Floating-point arithmetic, Polynomial evaluation, Horner's algorithm, Numerical integration.
\end{keywords}

%
%


\section{INTRODUCTION}

In floating-point arithmetic, rounding errors occur because of the finite representation of floating-point numbers in computers. The accumulation of rounding errors can significantly reduce the accuracy of a computation~\cite[p.~8]{parker1997monte}.

Stochastic arithmetic proposed in the 1950s by von Neumann and Goldstine~\cite{Neumann1947NumericalIO} is a computing paradigm developed as a model for exact computation with imprecise data. However, hardware units proposing stochastic rounding (SR) are still unavailable in most computer machines. Various specialized processors have introduced it such as: Graphcore IPUs which supports SR for binary32 and binary16 arithmetic~\cite{graph-a, graph-b, graph-c}, or Intel neuromorphic chip Loihi~\cite{davies2018loihi} to improve the accuracy of biological neuron and synapse models. Moreover, several related patents are owned by major chip designers such as AMD~\cite{amd}, NVIDIA~\cite{nvidia}, IBM ~\cite{ibm1,ibm2}, and other computing companies~\cite{com1,com2,com3}. These developments support the idea of hardware implementations using stochastic rounding becoming more available in the future.

Stochastic arithmetic has two main applications~\cite{survey}. 
First, it can be used to estimate empirically the numerical error of complex programs. In this context, stochastic arithmetic introduces a random noise in each floating-point operation to simulate the effect of rounding errors. To make this error analysis automatic, various tools such as Verificarlo~\cite{verificarlo}, Cadna~\cite{cadna} and Verrou~\cite{verrou} have been developed.
Second, stochastic arithmetic is sometimes used as a replacement for the default deterministic rounding mode in numerical simulations. Indeed, it has been demonstrated that in multiple domains, SR provides positive effects compared to the deterministic IEEE-754~\cite{norm} default rounding mode.

In particular, in the neural networks field, using SR instead of deterministic rounding~\cite{gupta} enables training with smaller data types and better accuracy. 
Indeed, the noise during training has been observed to help regularization and avoid model over-fitting~\cite{ml1,ml2}. 
The positive effect of SR extends also to the calculation of the solution of ordinary differential
equations (ODEs) in low precision~\cite{ode1,ode2} where SR reduces the accumulation of rounding errors by avoiding stagnation phenomenon when the step decreases.
Various other applications such as PDEs, Quantum mechanics, Quantum computing use SR to improve their results~\cite{survey}.

The IEEE-754 norm defines five rounding modes for floating-point arithmetic which are all deterministic: round to nearest ties to even (default), round to nearest ties away, round to zero, round to $+ \infty$, and round to $-\infty$. In section~\ref{sec:back}, we present a floating-point arithmetic background, and, we describe two stochastic rounding modes defined in \cite[p.~34]{parker1997monte}: SR-nearness and SR-up-or-down.
Section~\ref{sec:ODE} presents our first contribution: we study the bias and we compare the two stochastic rounding modes above and the default rounding mode in the IEEE-754 norm (RN-nearest32) on rectangular integration, which is at the basis of Euler's Method for ODE. We show that, contrarily to SR-up-or-down, SR-nearness is unbiased. An exact expression and an estimation of the bias are given for SR-up-or-down. We show how the accumulation of errors with both SR-up-or-down and IEEE-754 modes leads to results significantly less accurate than with SR-nearness.

At the beginning of section~\ref{sec:pbBound}, a probabilistic background, as well as some properties of SR-nearness rounding are presented. We recall some techniques that have been used in order to obtain a probabilistic bound for the inner product on $\sqrt{n}$ instead of the deterministic bound on $n$. We explain two approaches proposed by Higham and Mary~\cite{theo19} and Ilse, Ipsen and Zhou~\cite{ilse}. We extend these techniques for our second contribution and showing that they can still be used in situations where one multiplication operand is affected by an error.
In particular, using SR-nearness and without additional assumption, we provide a probabilistic bound in $O(\sqrt{n})$ rather than the deterministic bound which is in $O(n)$. We conclude this section with numerical experiments comparing the bounds above for the Chebyshev polynomial.

\section{Background}
\label{sec:back}
\subsection{Floating-point arithmetic}\label{sec:FP_def}
A normal floating-point number in such a format is a number $x$ for which there exists a triple $(s, m, e)$ such that $x= \pm m \times \beta^{e-p}$, where $\beta$ is the basis, $e$ is the exponent, $p$ is the working precision, and $m$ is an integer (the significand) such that $\beta^{p-1} \leq m < \beta^p$. This triple is unique.
We only consider normal floating-point numbers; detailed information on the floating-point format most generally in use in current computer systems is defined in the IEEE-754 norm~\cite{norm}.

Let us denote $\mathcal{F}\subset \mathbb{R}$ the set of normal floating-point numbers and $x\in \mathbb{R}$. Upward rounding $\lceil x \rceil $ and downward rounding $\lfloor x \rfloor$ are defined by:
$$ \lceil x \rceil=\min\{y\in \mathcal{F} : y \geq x\},  \quad \lfloor x \rfloor=\max\{y\in \mathcal{F} : y \leq x\},$$
clearly, $\lfloor x \rfloor  \leq x \leq \lceil x \rceil$, with equalities if and only if $x \in \mathcal{F}$.  

The floating-point approximation of a real number $x\ne 0$ is one of $\lfloor x\rfloor$ or $\lceil x\rceil$.
\begin{equation}
    \fl(x) =x(1+\delta), \label{fl(x)} 
\end{equation}
where $\delta = \frac{ \fl(x) - x}{x}$ is the relative error: $\lvert \delta \rvert \leq \beta^{1-p}$.
In the following, we note $u=\beta^{1-p}$. IEEE-754 mode RN (round to nearest, ties to even) has ~the ~stronger ~property\footnote{In many works focusing on IEEE-754 RN, $u$ is chosen to be $\frac12\beta^{1-p}$.} that ~$\lvert \delta \rvert \leq\frac12\beta^{1-p}=\frac12u$. 

For $x, y\in\mathcal F$, the considered rounding modes verify  $\fl(x\op y)\in\{\lfloor x\op y\rfloor, \lceil x\op y\rceil\}$ for $\op\in\{+, -, *, /\}$. Moreover, for IEEE-754 RN~\cite{norm} and stochastic rounding~\cite{theo21stocha} the error in one operation is bounded:
\begin{equation}
     \fl(x \op y) = (x \op y)(1+\delta), \; \lvert \delta \rvert \leq u, \label{fl(xopy)}
 \end{equation}
specifically for RN we have $\lvert \delta \rvert \leq \frac12 u$.

Assume that $x$ is a real that is not representable: $x\in \mathbb{R} \setminus \mathcal{F}$. 
The machine-epsilon or the distance between the two floating-point numbers enclosing $x$ is
$\epsilon(x) = \lceil x \rceil - \lfloor x \rfloor = \beta^{e-p}$. 
The fraction of $\epsilon(x)$ rounded away, as shown ~in ~figure~\ref{fig:theta}, is 
~~$\theta(x) = \frac{x - \lfloor x \rfloor}{\lceil x \rceil - \lfloor x \rfloor }$

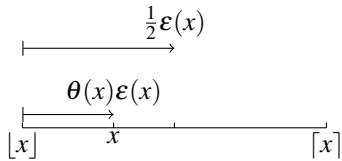
\begin{figure}
\centering
\begin{tikzpicture}[xscale=4]
\draw (0,0) -- (1,0);
\draw[shift={(0,0)},color=black] (0pt,0pt) -- (0pt, 2pt) node[below] {$\lfloor x \rfloor$};
\draw[shift={(1,0)},color=black] (0pt,0pt) -- (0pt, 2pt) node[below] {$\lceil x \rceil$};
\draw[shift={(.3,0)},color=black] (0pt,0pt) -- (0pt, 2pt) node[below] {$x$};
\draw[shift={(.5,0)},color=black] (0pt,0pt) -- (0pt, 2pt);
\draw[|->] (0, 30pt) -- (.5, 30pt) node[above] {$\frac{1}{2}\epsilon(x)$};
\draw[|->] (0, 5pt) -- (.3, 5pt) node[above] {$\theta(x)\epsilon(x)$};
\end{tikzpicture}

    \caption{$\theta(x)$ is the fraction of $\epsilon(x)$ to be rounded away.}
    \label{fig:theta}
\end{figure}

We note $\llfloor x \rrfloor$ the integer part of $x$. The following lemma gives an important property of downward rounding.
\begin{lem}\label{integer part}
Let $x\in \mathbb{R} \setminus \mathcal{F}$. $\beta^{p-e}\lfloor x \rfloor = \llfloor \beta^{p-e} x \rrfloor$.
\end{lem}

\begin{proof}
We know that $\beta^{p-e}\lfloor x \rfloor,\beta^{p-e} \lceil x \rceil \in \mathbb{Z}$, and
$\lfloor x \rfloor < x < \lceil x \rceil$, then $\beta^{p-e}\lfloor x \rfloor <\beta^{p-e} x < \beta^{p-e} \lceil x \rceil$. We thus have
$$\beta^{p-e}\lfloor x \rfloor \leq \llfloor \beta^{p-e} x \rrfloor < \beta^{p-e} \lceil x \rceil.$$
Since $\lceil x \rceil -\lfloor x \rfloor = \beta^{e-p}$, then  $\beta^{p-e} \lceil x \rceil - \beta^{p-e}  \lfloor x \rfloor =1$ and
$$\beta^{p-e}\lfloor x \rfloor \leq \llfloor \beta^{p-e} x \rrfloor <  \beta^{p-e} \lfloor x \rfloor +1.$$
\end{proof}

\subsection{Stochastic arithmetic}\label{sec:StochaticRound_def}

In order to define the two principal stochastic rounding modes, consider $\widehat x$ the random variable of the distribution of results after random rounding of $x$. Then
$$\widehat x = \textbf{random\_round}(x) = \textbf{round}(x + \beta^{e-p}\xi),$$
where $\xi$ is a random variable that can be discrete or continuous and \textbf{round} is the default IEEE-754 rounding mode to the nearest.
\begin{itemize}
    \item SR-nearness~\cite[p.~34]{parker1997monte} is defined for an $\xi$ uniform random variable on $]-\frac12 ; \frac{1}{2}[$, hence $\xi$ has mean $0$ and standard deviation $1/\sqrt{12}$.
            \begin{figure}
    \centering
\begin{tikzpicture}[xscale=5]
\draw (0,0) -- (1,0);
\draw[shift={(0,0)},color=black] (0pt,0pt) -- (0pt, 2pt) node[below] {$\lfloor x \rfloor$};
\draw[shift={(1,0)},color=black] (0pt,0pt) -- (0pt, 2pt) node[below] {$\lceil x \rceil$};
\draw[shift={(.3,0)},color=black] (0pt,0pt) -- (0pt, 2pt) node[below] {$x$};
\draw (0,0) .. controls (.15,.4) .. (.3,0)  (0.15,0.5) node {$1-\theta(x)$};
\draw (.3,0) .. controls (.65,.7) .. (1,0) (0.65,0.8) node {$\theta(x)$} ;
\end{tikzpicture}
    \caption{\textbf{SR-nearness}.}
\end{figure}
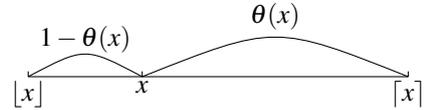

In other words, SR-nearness consists ~in ~rounding ~up ~$x\in \mathbb{R} \setminus \mathcal{F}$ with probability $\theta(x)=(x-\lfloor x \rfloor)/(\lceil x \rceil - \lfloor x \rfloor )$, we round down with probability $1-\theta(x)$, proportional to the distances between $x$ and the closest representable numbers. 

Stott Parker shows that SR-nearness is unbiased \cite[p.~34]{parker1997monte} so $\mathbb{E}(\widehat x) = x$. 
    \item SR-up-or-down is defined for a $\xi$ uniform random variable on $]-\theta(x); 1-\theta(x)[$,  $\xi$ is biased with mean $ \frac{1}{2}-\theta(x)$ and standard deviation $1/\sqrt{12}$.
    \begin{figure}
   \centering
\begin{tikzpicture}[xscale=5]
\draw (0,0) -- (1,0);
\draw[shift={(0,0)},color=black] (0pt,0pt) -- (0pt, 2pt) node[below] {$\lfloor x \rfloor$};
\draw[shift={(1,0)},color=black] (0pt,0pt) -- (0pt, 2pt) node[below] {$\lceil x \rceil$};
\draw[shift={(.3,0)},color=black] (0pt,0pt) -- (0pt, 2pt) node[below] {$x$};
\draw (0,0) .. controls (.15,.4) .. (.3,0) (0.15,0.6) node {$\frac12$};
\draw (.3,0) .. controls (.65,.7) .. (1,0) (0.65,0.85) node {$\frac12$} ;
\end{tikzpicture}
    \caption{\textbf{SR-up-or-down}.}
\end{figure}
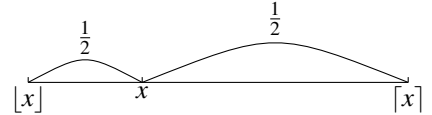

In other words, SR-up-or-down consists in rounding $x$ up or down with probability $\frac12$. 
SR-up-or-down mode can be expressed~\cite[p.~34]{parker1997monte} in terms of $\theta(x)$:
since the two outcomes of SR-up-or-down mode are equiprobable, we have $\mathbb{E}(\widehat x) = \frac{\lceil x \rceil + \lfloor x \rfloor}{2}$, which allow us to write the bias as
$$ \mathbb{E}(\widehat x -x) = \frac{\lceil x \rceil + \lfloor x \rfloor}{2} - x,$$
because $ \theta(x)= \frac{x-\lfloor x \rfloor}{\lceil x \rceil - \lfloor x \rfloor}$
\begin{align*}
    \mathbb{E}(\widehat x-x) &= (\lceil x \rceil - \lfloor x \rfloor)(\frac{1}{2}-\theta(x)) \\
     &= \epsilon(x)(\frac{1}{2}-\theta(x)).
\end{align*} 
Thus, we conclude that SR-up-or-down is biased and the expected value depends on $\theta(x)$ and $\epsilon(x)$.
\end{itemize}

\section{Integrating a constant function}
\label{sec:ODE}
Rectangular integration rule is a classic approximation for performing numerical integration: the area under a curve is approximated by a sum of $N$ rectangle areas:

$$ \int_a^b f(t) dt \approx \sum^{N-1}_{k=0} h f(a+kh)  $$
where $h=\frac{b-a}{N}$.
In particular, rectangular rule is one of the resolution techniques for ODE using Euler’s forward method.

Verrou's tutorial~\cite{verrou-tutorial} integrates the cosinus function with the rectangular rule; with deterministic round to nearest or SR-up-or-down modes, the solution is biased. When the number of integration steps grows, this bias can become high and degrade the quality of the solution.
In this section, we show why deterministic and SR-up-or-down modes are sometimes biased with rectangular rule. 

We perform the analysis on a constant function  $f(t) = 1$ for all $t\in [0;1]$. With $f$ constant, the evaluation error is zero, making it clear how the numerical error accumulates on the summation.

Denote $x=1=\sum_{k=0}^{N-1} h$, where $h=1/N$. The distribution $\widehat x$ is produced by summing $N$ times the integration step $h$.
We note $\widehat s_k$ the random variable for the partial sum at step $0 \leq k \leq N-1$ and $s_k$ the exact expected result,
with $\widehat s_{N-1} = \widehat x$.

\paragraph{SR-up-or-down}: As shown before, for each $\widehat s_k$ we introduce a bias corresponding to
$$
\mathbb{E}(\widehat s_k-s_k) = \epsilon(s_k)(\frac{1}{2}-\theta(s_k)),
$$ 
from the definition of $\theta(s_k)$, we have $0< \theta(s_k)< 1$, then $-\frac12<\frac12 - \theta(s_k) < \frac12$ and

$$
\lvert \mathbb{E}(\widehat s_k-s_k) \rvert < \frac12 \epsilon(s_k).
$$
Table~\ref{tab:skbias} shows these different values for $N=20$.

\begin{table}
    \centering
     \caption{$s_k$, $\theta$, bias and $\epsilon$ for $N=20$.}
    \begin{tabular}{crrrr}
    \toprule
         $k$ & $s_k$ & $\theta(s_k)$ & $\mathbb{E}(S_k-s_k)$ & $\epsilon(s_k)$ \\
    \midrule
2 & 0.150...  & 0.7500 & -3.725290e-09 & 1.490116e-08 \\
3 & 0.200... & 0.2500 &  3.725290e-09 & 1.490116e-08 \\
\\
4 & 0.250...  & 0.6250 & -3.725290e-09 & 2.980232e-08 \\
5 & 0.300...  & 0.6250 & -3.725290e-09 & 2.980232e-08 \\
6 & 0.350...  & 0.6250 & -3.725290e-09 & 2.980232e-08 \\
7 & 0.400...  & 0.6250 & -3.725290e-09 & 2.980232e-08 \\
8 & 0.450...  & 0.6250 & -3.725290e-09 & 2.980232e-08 \\
\\
9 & 0.500...  & 0.3125 &  1.117587e-08 & 5.960464e-08 \\
10 & 0.550... & 0.8125 & -1.862645e-08 & 5.960464e-08 \\
11 & 0.600... & 0.8125 & -1.862645e-08 & 5.960464e-08 \\
12 & 0.650... & 0.8125 & -1.862645e-08 & 5.960464e-08 \\
13 & 0.700... & 0.8125 & -1.862645e-08 & 5.960464e-08 \\
14 & 0.749... & 0.8125 & -1.862645e-08 & 5.960464e-08 \\
15 & 0.799... & 0.8125 & -1.862645e-08 & 5.960464e-08 \\
16 & 0.849... & 0.8125 & -1.862645e-08 & 5.960464e-08 \\
17 & 0.899... & 0.8125 & -1.862645e-08 & 5.960464e-08 \\
18 & 0.949... & 0.8125 & -1.862645e-08 & 5.960464e-08 \\
19 & 0.999... & 0.8125 & -1.862645e-08 & 5.960464e-08 \\
\bottomrule
    \end{tabular}
   
    \label{tab:skbias}
\end{table}

Interestingly, in this table, we note that $\theta(s_k)$ is constant
between two powers-of-the-base except for the first value. For example for $9<k<20$, $s_k$ stays within $[2^{-1}; 2^{0})$ and both $\theta(s_k)$ and $\mathbb{E}(S_k-s_k)$ are constant.  Let us show why that is always the case.

Suppose $s_k \in [\beta^{e}; \beta^{e+1})$. Then $\epsilon(s_k) = \beta^{e-p}$. 
At each step the next partial sum is computed as, $s_{k+1} =\fl(s_k) + h$, in that case, using the lemma \ref{integer part}, we have

\begin{eqnarray*}
\theta(s_{k+1}) &=& \beta^{p-e}( s_{k+1} - \lfloor s_{k+1} \rfloor)\\
&=& \beta^{p-e}s_{k+1} - \llfloor \beta^{p-e}s_{k+1} \rrfloor \\ 
&=& \beta^{p-e}\fl(s_k) + \beta^{p-e}h  - \llfloor \beta^{p-e}\fl(s_k) + \beta^{p-e}h \rrfloor.
\end{eqnarray*}

Since $\fl(s_k) \in \mathcal{F}$, we have $\beta^{p-e}\fl(s_k) \in \mathbb{Z}$ and
$$\llfloor \beta^{p-e}\fl(s_k) + \beta^{p-e}h \rrfloor = \beta^{p-e}\fl(s_k) + \llfloor  \beta^{p-e}h \rrfloor. $$ 
Finally 
$$\theta(s_{k+1}) =  \beta^{p-e}h - \llfloor \beta^{p-e}h \rrfloor= \theta(h).$$ 

Thus $\theta(s_{k+1})$ depends only on $h$ and $e$. Recursively for all $l>0$ satisfying $s_{k+l} \in [\beta^{e}; \beta^{e+1})$, $\theta(s_{k+l}) = \beta^{p-e}h - \llfloor \beta^{p-e}h \rrfloor = \theta(h)$ is constant. The bias
\begin{eqnarray*}
\mathbb{E}(\widehat s_k-s_{k+l}) &=& \epsilon(s_k)(\frac{1}{2}-\theta(s_{k+l})) \\
&=& \beta^{e-p}(\frac{1}{2}- \beta^{p-e}h - \llfloor \beta^{p-e}h \rrfloor),
\end{eqnarray*}
is also constant in this interval.

Between two powers of the base, $\theta$ remain constant as well as the bias. Because the bias is constant (and, consequently, its sign too), it accumulates across iterations.

The total bias can be written as,

$$
\mathbb{E}( \widehat x - x) = E( \widehat s_{N-1} - s_{N-1}) = \epsilon(s_{N-1})(\frac{1}{2} - \theta(s_{N-1})), 
$$
and 
$$
   \lvert \mathbb{E}( \widehat x - x) \rvert < \frac12 \epsilon(s_{N-1}).
$$

We can neglect the effect of the first partial sum in each power-of-the-base interval.
Since $s_k$ increases arithmetically, and the size of the power of the base interval grows geometrically, the bias in the last interval usually dominates since it contains more summation terms and has a larger $\epsilon$. 

\textbf{Numerical experiment.}
We have verified numerically that the above expression for the bias closely predicts the bias measured with SR-up-or-down stochastic rounding. 

We consider a fixed number of iterations N. We ran one time the C program in listing~\ref{lst:constant-fixed} with each of the two previously defined stochastic rounding modes as well as round to nearest. 
Figure~\ref{fig:y=1} plots the three distributions over $N$. 

\begin{lstfloat}
\begin{verbatim}
float h = 1/N;
float s = 0.0;
for (int i=0 ; i < N ; i++) { 
    s += h*1; 
}
return s;
\end{verbatim}
    \caption{Fixed-step rectangle integration of a constant}
    \label{lst:constant-fixed}
\end{lstfloat}

\begin{figure}
    \centering
    \includegraphics[scale=0.75]{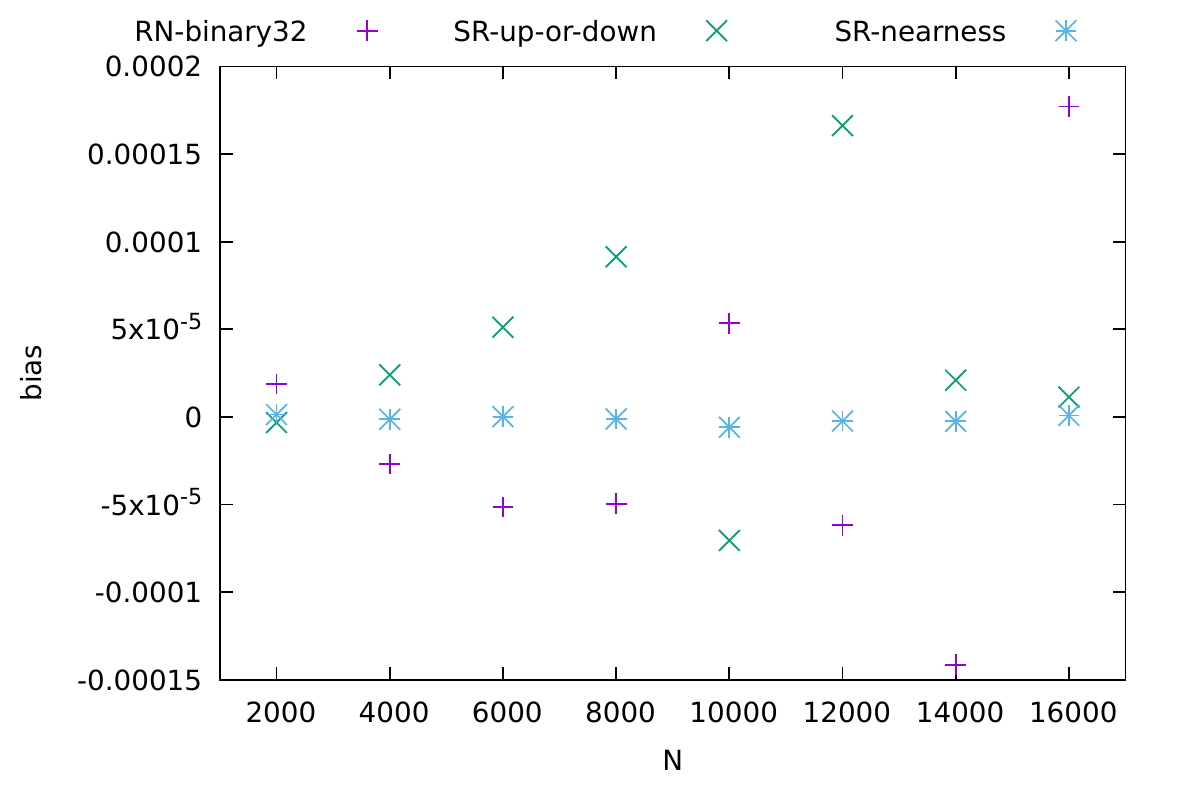}
    \caption{Round to nearest (RN-binary32) vs stochastic rounding SR-up-or-down and SR-nearness for $\int_{0}^{1} 1 \,dt$. }
    \label{fig:y=1}
\end{figure}

Figure~\ref{fig:y=1} shows that SR-nearness mode samples is unbiased regardless of the number $N$ of rectangles. The unbiased nature is unsurprising, since SR-nearness mode can be seen as a sub case of Monte Carlo Arithmetic (MCA). Stott Parker proves~\cite[p.~46]{parker1997monte} that the expectation of a sum of terms with MCA is the exact mathematical result.

On the other hand, SR-up-or-down mode and RN-binary-32 samples have a bias, which confirms the previous results for SR-up-or-down mode. The maximal amplitude of the bias for both SR-up-or-down and increases with $N$ because of errors accumulation. The bias is reproducible and constant across different runs.

In conclusion, this example illustrates that SR-nearness is unbiased not only for one elementary operation, but, even in other numerical methods such as rectangular integration, it is much closer to the expected value than SR-up-or-down or RN-binary32, in particular for $N$ large. 
In the remainder of this paper, we focus on SR-nearness.

\section{Probabilistic bound}
\label{sec:pbBound}
The aim of this section is to introduce a probabilistic bound in $O(\sqrt{n} u)$ on the forward error of Horner algorithm, based on the Azuma–Hoeffding inequality and martingale properties. 
We first recall some probabilistic properties, then provide a review of the results made on the forward error in the numerically computed inner product, and we conclude with numerical experiments illustrating the previous results.

\subsection{Probabilistic Background}
A random variable $Y$ is said to be mean independent of random variable $X$ if and only if its conditional mean ~~$\mathbb{E}[Y/ X=x]$ equals its unconditional mean $\mathbb{E}(Y)$ for all reals $x$ such that the probability that $X = x$ is not zero and we write $\mathbb{E}[Y / X]=\mathbb{E}(Y)$. The random sequence $(X_1, X_2, \ldots)$ is mean independent if $\mathbb{E}[X_k / X_1, . . . , X_{k-1}] = \mathbb{E}(X_k)$ for all $k$.

\begin{pro}
Let $X$ and $Y$ two real random variables: $X$ and $Y$ are independents $\implies X$ is mean independent from $Y$ $\implies X$ and $Y$ are uncorrelated. The reciprocals of these two implications are false.
\end{pro}

Under stochastic rounding, the elementary operations are stochastically rounded, hence, for $x \in \mathbb{R}$, if $\fl(x)=x(1+\delta)$ is obtained by SR-nearness then $\delta$ is a random variable such that $\mathbb{E}(\delta)=0.$ For $x_1, x_2, x_3 \in \mathbb{R}$, such that $c=x_1 \op x_2 \op x_3$ where $\op \in \{+, -, *, /\}$, and 
$$ \fl(c)= \big((x_1 \op x_2)(1+\delta_1) \op x_3\big)(1+\delta_2)$$
obtained from SR-nearness. $\delta_1, \delta_2$ are random variables such that $\mathbb{E}(\delta_1)=\mathbb{E}(\delta_2)=0$.

The following lemma has been proven in~\cite[Lem 5.2]{theo21stocha} and shows that SR-nearness satisfies the property of mean independence.
\begin{lem}\label{meanind}
For some $\delta_1, \delta_2, . . .$, in that order obtained from SR-nearness, the $\delta_k$ are random variables with mean zero such that $\mathbb{E}[\delta_k /  \delta_1, . . . , \delta_{k-1}] = \mathbb{E}(\delta_k)= 0$.
\end{lem}
Finally, we need to recall the concept of a martingale and the Azuma-Hoeffding inequality for a martingale~\cite{azum}. 

\begin{defn}
\label{def:martingale}
A sequence of random variables $M_1, ... , M_n$ is a martingale with respect to the sequence $X_1, . . . , X_n$ if, for all $k,$
\begin{itemize}
    \item $M_k$ is a function of $X_1, ..., X_k$,
    \item $\mathbb{E}(\lvert M_k \rvert ) < \infty,$ and
    \item $\mathbb{E}[M_k /  X_1, ..., X_{k-1}]=M_{k-1}$.
\end{itemize}

\end{defn}

\begin{lem}(Azuma-Hoeffding inequality). Let $M_0, ..., M_n$ be a martingale ~with ~respect to ~a ~sequence $X_1, . . . , X_n.$ We assume that there exist $a_k<b_k$ such  that  $a_k \leq M_k - M_{k-1} \leq  b_k$ for $k = 1: n.$ Then for
any $A  > 0$ 
$$ \mathbb{P}(\lvert M_n - M_0 \rvert \geq A) \leq 2 \exp \left( 
-\frac{2A^2}{\sum_{k=1}^n(b_k-a_k)^2} 
\right).$$
In the particular case $a_k=-b_k$ we have 
$$ \mathbb{P}\left( \lvert M_n - M_0 \rvert \geq \sqrt{\sum_{k=1}^n b_k^2} \sqrt{2 \ln (2 / \lambda)} \right) \leq \lambda,$$
where $0< \lambda <1$.
\end{lem}

\subsection{Inner product bound}
Under SR-nearness, the deterministic bound on the error of an inner product $y=a^{\top}b,$ where $a,b\in \mathbb{R}^n$, is proportional to $n u$, where $u$ is the unit roundoff of the floating-point arithmetic in use. Wilkinson~\cite[sec 1.33]{wilk} had the intuition that the roundoff error accumulated in $n$ operations is proportional to $\sqrt{n} u$ rather than $n u$. Based on the mean independence of errors established in Lemma~\ref{meanind}, Higham and Mary~\cite{theo19} and Ilse, Ipsen and Zhou~\cite{ilse} have proved this result for SR. Both works build on the mean independence property of SR. This allows them to form a martingale, and then to apply the Azuma-Hoeffding concentration inequality. The difference between these two works is in the way they form the martingale. In \cite{theo19}, the martingale is built using the errors accumulated in the whole process, while \cite{ilse} forms it by following step-by-step how the error accumulates. In particular, they distinguish between the multiplications and additions computed in the inner product, and carefully monitor their mean independences. In the following, we adapt this construction to Horner's algorithm. 

\subsection{Horner algorithm bound}

Horner algorithm is an efficient way of evaluating polynomials. When performed in floating-point arithmetic this algorithm may suffer from catastrophic cancellations and yield a computed value with less accurate than expected.

In the following, we derive a probabilistic bound for the computed $\fl(P(x))$ based on the previous methods applied for the inner product. 

\begin{mode}\label{mod}
Let $P(x) = \sum_{i=0}^n a_i x^i$, Horner's rule consists in writing this polynomial as 
$$P(x)= \big(((a_nx +a_{n-1})x +a_{n-2})x ... +a_1\big)x +a_0.$$

We define by induction the following sequence 
\begin{eqnarray*}
r_0 &=& a_n, \\
r_{2k-1} &=& r_{2k-2}x, \\
r_{2k} &=& r_{2k-1} +a_{n-k},
\end{eqnarray*}
for all $1 \leq k \leq n$. Likewise, We define 
\begin{eqnarray*}
\widehat r_0 &=& a_n, \\
\widehat r_{2k-1} &=& \widehat r_{2k-2}x (1+\delta_{2k-1}),\\  
\widehat r_{2k} &=& (\widehat r_{2k-1} +a_{n-k})(1+\delta_{2k}),
\end{eqnarray*}
for all $1 \leq k \leq n$, with $\delta_{2k-1}$ and $\delta_{2k}$ represent the rounding errors from the products and the additions, respectively.

\end{mode}

Let us define $Z_i :=\widehat r_i - r_i$ for all $1 \leq i \leq 2n$. The total forward error is $\lvert Z_{2n} \rvert = \lvert \widehat r_{2n} -  r_{2n} \rvert = \lvert \fl(P(x))  - P(x) \rvert$. The computation of $Z_{2n}$ introduces $\delta_1,..., \delta_{2n}$ such that $\lvert \delta_k \rvert \leq u $ for all $1 \leq k \leq 2n$. We prove by induction that the accumulation of the $(1+\delta_k)$ errors on the $a_i x^i$ term for $0 \leq i \leq n-1$ reaches $\prod_{k=2n -2i}^{2n} (1+\delta_k):=\varphi_i$, and for $i=n$, it reaches $\prod_{k=1}^{2n} (1+\delta_k):=\varphi_n$. Hence

\begin{eqnarray*}
\lvert \fl(P(x))  - P(x) \rvert &=&  \left\lvert \sum_{i=0}^n a_i x^{i} (\varphi_i-1)  \right\rvert \leq \sum_{i=0}^n \lvert  a_i x^i\rvert \lvert  \varphi_i-1  \rvert\\
&\leq&  \sum_{i=0}^n \lvert  a_i x^i\rvert  \gamma_{2n},
\end{eqnarray*}
where $\gamma_{2n} = (1+u)^{2n} -1 = 2nu +O(u^2)$ (we recall that, $\forall k$, $|\delta_k|\leq u$).
Finally 

\begin{equation}\label{detbound}
    \frac{\lvert \fl(P(x))  - P(x) \rvert}{\lvert P(x) \rvert} \leq cond_1(P,x) \gamma_{2n} ,
\end{equation}
where  $cond_1(P,x) := \frac{\sum_{i=1}^n \lvert a_i x^i \rvert}{\lvert P(x) \rvert}$ is the condition number of the evaluation of $P$ in $x$ using the 1-norm. The deterministic bound is proportional to $nu$. In the following, we prove a probabilistic bound in $O(\sqrt{n}u)$. 

The partial sums forward errors satisfy
\begin{eqnarray*}
 Z_{2k-1} &=& \widehat r_{2k-1} - r_{2k-1}\\
 &=& \widehat r_{2k-2}x (1+\delta_{2k-1}) - r_{2k-2}x\\
 &=& xZ_{2k-2} +\widehat r_{2k-2}x \delta_{2k-1},  \\
 Z_{2k} &=& \widehat r_{2k} - r_{2k}\\
 &=& (\widehat r_{2k-1} +a_{n-k})(1+\delta_{2k}) - r_{2k-1} - a_{n-k}\\
&=& \widehat r_{2k-1} + (\widehat r_{2k-1} +a_{n-k})\delta_{2k}- r_{2k-1}\\
&=& Z_{2k-1} +(\widehat r_{2k-1} +a_{n-k}) \delta_{2k}, 
\end{eqnarray*}
for all $1 \leq k \leq n$.
The sequence $Z_1,...,Z_{2n}$ does not form a martingale with respect to $\delta_1,...,\delta_{2n}$ because, due to the multiplication in odd steps, $ E[Z_{2k-1}/\delta_1,...,\delta_{2k-2}]= xZ_{2k-2}$. In order to form a martingale and use the Azuma-Hoeffding inequality, we define the following variable change  
$$ Y_i = \frac{Z_i}{x^{\llfloor (i+1)/2 \rrfloor}},$$
where $\llfloor (i+1)/2 \rrfloor$ is the integer part of $(i+1)/2$, we thus have
 \begin{equation}\label{eq1}
        Y_{2k-1} = Y_{2k-2} + \frac{1}{x^{k-1}}\widehat r_{2k-2} \delta_{2k-1},\\    
 \end{equation}
 \begin{equation}\label{eq2}
        Y_{2k}  = Y_{2k-1} + \frac{1}{x^k}(\widehat r_{2k-1} + a_{n-k})\delta_{2k},
 \end{equation}
for all $1 \leq k \leq n.$

\begin{thm}
The sequence of random variables $Y_1,...,Y_{2n}$ is a martingale with respect to $\delta_1, ..., \delta_{2n}$.
\end{thm}

\begin{proof}
We check that the three conditions of definition~\ref{def:martingale} are satisfied. Throughout ~the ~proof, ~we ~note ~the ~set ~$\mathbb{F}_k= \{\delta_1, ..., \delta_k\}$.

\begin{itemize}
   \item The recursion in Model \ref{mod} shows that $Y_i$ is a function of $\delta_1, ..., \delta_{i}$ for all $1 \leq i \leq 2n$.
   
   \item $\mathbb{E}(\lvert Y_i \rvert ) $ is finite because ~$x$ and $a_k$ are ~~finite ~for ~all ~~$n-i \leq k \leq n$ and $\lvert \delta_j \rvert \leq u$ for all $1 \leq j \leq i$.
    
    \item We prove that $\mathbb{E}[Y_i /\mathbb{F}_{i-1}] = Y_{i-1}$ by distinguishing the even and odd case. \\
    Firstly, using the mean independence of $\delta_1, ... \delta_{2k-1}$ and equation (\ref{eq1}) we obtain
    \begin{eqnarray*}
    \mathbb{E}[Y_{2k-1}/\mathbb{F}_{2k-2}] &=& \mathbb{E}[Y_{2k-2}/\mathbb{F}_{2k-2}] \\
    && + \mathbb{E}[ \frac{1}{x^{k-1}}\widehat r_{2k-2} \delta_{2k-1}/\mathbb{F}_{2k-2}]\\
    &=& Y_{2k-2} 
     + \frac{1}{x^{k-1}}\widehat r_{2k-2} \mathbb{E}[\delta_{2k-1}/\mathbb{F}_{2k-2}]\\
    &=& Y_{2k-2}.
    \end{eqnarray*}
    \end{itemize}
    
Secondly, using the mean independence of $\delta_1, ... \delta_{2k}$ and equation (\ref{eq2}) we obtain
    \begin{eqnarray*}
    \mathbb{E}[Y_{2k}/\mathbb{F}_{2k-1}] &=& \mathbb{E}[Y_{2k-1}/\mathbb{F}_{2k-1}]\\
    && + \mathbb{E}[\frac{1}{x^k}(\widehat r_{2k-1} + a_{n-k})\delta_{2k}/\mathbb{F}_{2k-1}]\\
    &=& Y_{2k-1} \\
    && + \frac{1}{x^k}(\widehat r_{2k-1} + a_{n-k})\mathbb{E}[\delta_{2k}/\mathbb{F}_{2k-1}]\\ 
    &=& Y_{2k-1}.
    \end{eqnarray*}\end{proof}

\begin{lem}\label{azum}
The above martingale $Y_1,..., Y_{2n}$ satisfies 
$$ \lvert Y_i - Y_{i-1} \rvert \leq C_i u, \quad 1\leq i \leq 2n,$$
where 
$$
C_{2k-1} =  \lvert a_n \rvert (1+u)^{2k-2} + \sum_{j=1}^{k-1} \lvert a_{n-j} \rvert \lvert x \rvert^{-j}(1+u)^{2(k-j)-1},
$$
$$
C_{2k} = \lvert a_n \rvert (1+u)^{2k-1} + \sum_{j=1}^k \lvert a_{n-j} \rvert \lvert x \rvert^{-j}(1+u)^{2(k-j)},
$$ 
for all $1\leq k \leq n.$
\end{lem}

\begin{proof}

Note that $Y_0=0$, then $\lvert Y_1 - Y_0 \rvert = \lvert Y_1 \rvert = \lvert a_n \rvert $ and the equality holds for $C_1$. Using equation~(\ref{eq1})
$$
\lvert Y_{2k-1} - Y_{2k-2} \rvert  \leq \frac{1}{\lvert x \rvert^{k-1}} \lvert \widehat r_{2k-2} \rvert u.
$$
However
\begin{eqnarray*}
\lvert \widehat r_{2k-2} \rvert  &\leq &  \lvert \widehat r_{2k-3} \rvert (1+u) + \lvert a_{n-k+1}\rvert (1+u) \\
& \leq & \lvert \widehat r_{2k-4} \rvert \lvert x \rvert (1+u)^2 + \lvert a_{n-k+1}\rvert (1+u),
\end{eqnarray*}
by induction we obtain
$$ \lvert \widehat r_{2k-2} \rvert \leq \lvert a_n \rvert  \lvert x \rvert^k (1+u)^{2k-2} + \sum_{j=1}^{k-1} \lvert a_{n-j} \rvert \lvert x \rvert^{k-j}(1+u)^{2(k-j)-1}.$$
This completes the proof for $C_{2k-1}$ for all $1\leq k \leq n$. A similar approach can be applied to proving the same result for $C_{2k}$ for all $1\leq k \leq n$.
\end{proof}

We now have all the tools to state and then demonstrate the main result of this section:

\begin{thm}
Under SR-nearness, for all $0 < \lambda <1$ and with probability at least $1-\lambda$
\begin{equation}\label{probound}
    \frac{\lvert \fl(P(x))   - P(x) \rvert}{\lvert P(x) \rvert} \leq cond_1(P,x)  \sqrt{u \gamma_{4n}} \sqrt{\ln (2 / \lambda)},
\end{equation}
where $cond_1(P,x) = \frac{\sum_{i=1}^n \lvert a_i x^i \rvert}{\lvert P(x) \rvert}$ is the condition number of the polynomial evaluation and $\gamma_{4n}=(1+u)^{4n} -1$.
\end{thm}

\begin{proof}
Recall that $\lvert \widehat r_{2n} - r_{2n} \rvert = \lvert Z_{2n} \rvert = \lvert x^n \rvert \lvert Y_{2n} \rvert$. Therefore, $ Y_1,...,Y_{2n}$ is a martingale with respect to $\delta_1, ..., \delta_{2n}$ and Lemma \ref{azum} implies $ \lvert Y_i - Y_{i-1} \rvert \leq C_i u$ for all $1\leq i \leq 2n$. Using the Azuma-Hoeffding inequality yields
$$\mathbb{P}\left( \lvert Y_{2n} \rvert \leq u \sqrt{\sum_{i=1}^{2n}C_i^2}\sqrt{2 \ln (2 / \lambda)}\right) \geq 1-\lambda , $$
it follows that
$$ \lvert Z_{2n} \rvert \leq u \sqrt{\sum_{i=1}^{2n} (\lvert x \rvert^n C_i)^2}\sqrt{2 \ln (2 / \lambda)}, $$
with probability at least $1-\lambda$, where
\begin{eqnarray*}
\lvert x \rvert^n C_{2k} &=& \lvert a_n \rvert \lvert x \rvert^n (1+u)^{2k-1} + \sum_{j=1}^k \lvert a_{n-j} x^{n-j} \rvert(1+u)^{2(k-j)}\\
&\leq & (1+u)^{2k-1} \sum_{j=0}^k \lvert a_{n-j} x^{n-j} \rvert\\
&\leq& (1+u)^{2k-1} \sum_{j=0}^n \lvert a_{j} x^j \rvert,
\end{eqnarray*}
for all $1\leq k \leq n.$
Hence, $$ (\lvert x \rvert^n C_{2k})^2 \leq (1+u)^{2(2k-1)} \big(\sum_{j=0}^n \lvert a_{j} x^j \rvert \big)^2.$$
In a similar way 
$$ (\lvert x \rvert^n C_{2k-1})^2 \leq (1+u)^{2(2k-2)} \big(\sum_{j=0}^n \lvert a_{j} x^j \rvert \big)^2.$$
Thus 
\begin{eqnarray*}
\sum_{i=1}^{2n} (\lvert x \rvert^n C_i)^2 &\leq&  \big(\sum_{j=0}^n \lvert a_{j} x^j \rvert \big)^2 \sum_{i=0}^{2n-1} ((1+u)^{2})^i\\
&=& \big(\sum_{j=0}^n \lvert a_{j} x^j \rvert \big)^2 \frac{((1+u)^2)^{2n}-1}{(1+u)^2-1}\\
&=&  \big(\sum_{j=0}^n \lvert a_{j} x^j \rvert \big)^2  \frac{\gamma_{4n}}{u^2+2u}.
\end{eqnarray*}
As a result 
$$ \lvert \fl(P(x))  - P(x) \rvert = \lvert Z_{2n} \rvert \leq \sum_{j=0}^n \lvert a_{j} x^j \rvert  \sqrt{\frac{u \gamma_{4n}}{2+u}}  \sqrt{2 \ln (2 / \lambda)},$$
with probability at least $1-\lambda$.
Finally
$$ \frac{\lvert \fl(P(x))  - P(x) \rvert}{\lvert P(x) \rvert} \leq cond_1(P,x) \sqrt{u \gamma_{4n}} \sqrt{\ln (2 / \lambda)},$$
with probability at least $1-\lambda$.
\end{proof}

\begin{rem}
The bounds (\ref{detbound}) and (\ref{probound}) have the same condition number, but differ in another factor: $\gamma_{2n}$ for (\ref{detbound}) against $ \sqrt{u \gamma_{4n}} \sqrt{\ln (2 / \lambda)}$ for (\ref{probound}).

For ~$n$ ~such that ~~$2nu < 1$,~\cite[Lemma 3.1]{higham2002} ~~implies ~~$\gamma_{2n} \leq \frac{2nu}{1-2nu},$ it follows that for $4nu<1$
\begin{eqnarray*}
\sqrt{u \gamma_{4n}} &\leq& \sqrt{ \frac{4nu^2}{1-4nu}}
= u \sqrt{n} \frac{2}{\sqrt{1-4nu}}.
\end{eqnarray*}
For $n$ large, Taylor's formula implies $\gamma_n = nu + O(u^2)$ and
$\gamma_{2n} \approx 2nu$.
This approach can't be used for $\sqrt{\gamma_{4n}}$ because it's indeterminate in $0$. However, 

$$\lim_{u \to 0} \frac{ \sqrt{u \gamma_{4n}} }{\sqrt{n} u} =2 \Longleftrightarrow \sqrt{u \gamma_{4n}} \approx 2\sqrt{n} u.$$
Eventually, the probabilistic bound for the forward error of Horner's algorithm is in $O(\sqrt{n} u)$. 
\end{rem}

\subsection{Numerical experiments.}
In this section, we illustrate that the probabilistic bound is tighter than the deterministic bound for SR-nearness forward error on a numerical application: the evaluation of the ~Chebyshev ~polynomial. 
We use Horner’s method to evaluate the polynomial $P(x)=T_{N}(x) = \sum_{i=0}^{\llfloor \frac{N}{2} \rrfloor} a_i (x^2)^i$ where $T_{N}$ is the Chebyshev polynomial of degree $N$. Consider an even $N=2n$. We use single-precision (binary32) for both SR-nearness and round to nearest ties to even. All SR computations are repeated $30$ times with verificarlo~\cite{verificarlo}; we plot all samples and the forward error of the average of the 30 SR instances. The following error bounds and evaluations apply:

\begin{eqnarray*}
 \text{Probabilistic bound}    &=& cond_1(P,x) \sqrt{u \gamma_{4n}} \sqrt{\ln (2 / \lambda)}, \\
  \text{Deterministic bound}   &=& cond_1(P,x) \gamma_{2n},\\
  \text{SR-nearness} &=& \frac{\lvert \fl(P(x))  - P(x) \rvert}{\lvert P(x) \rvert}.
\end{eqnarray*}
\begin{figure}
    \centering
    \includegraphics[scale=0.6]{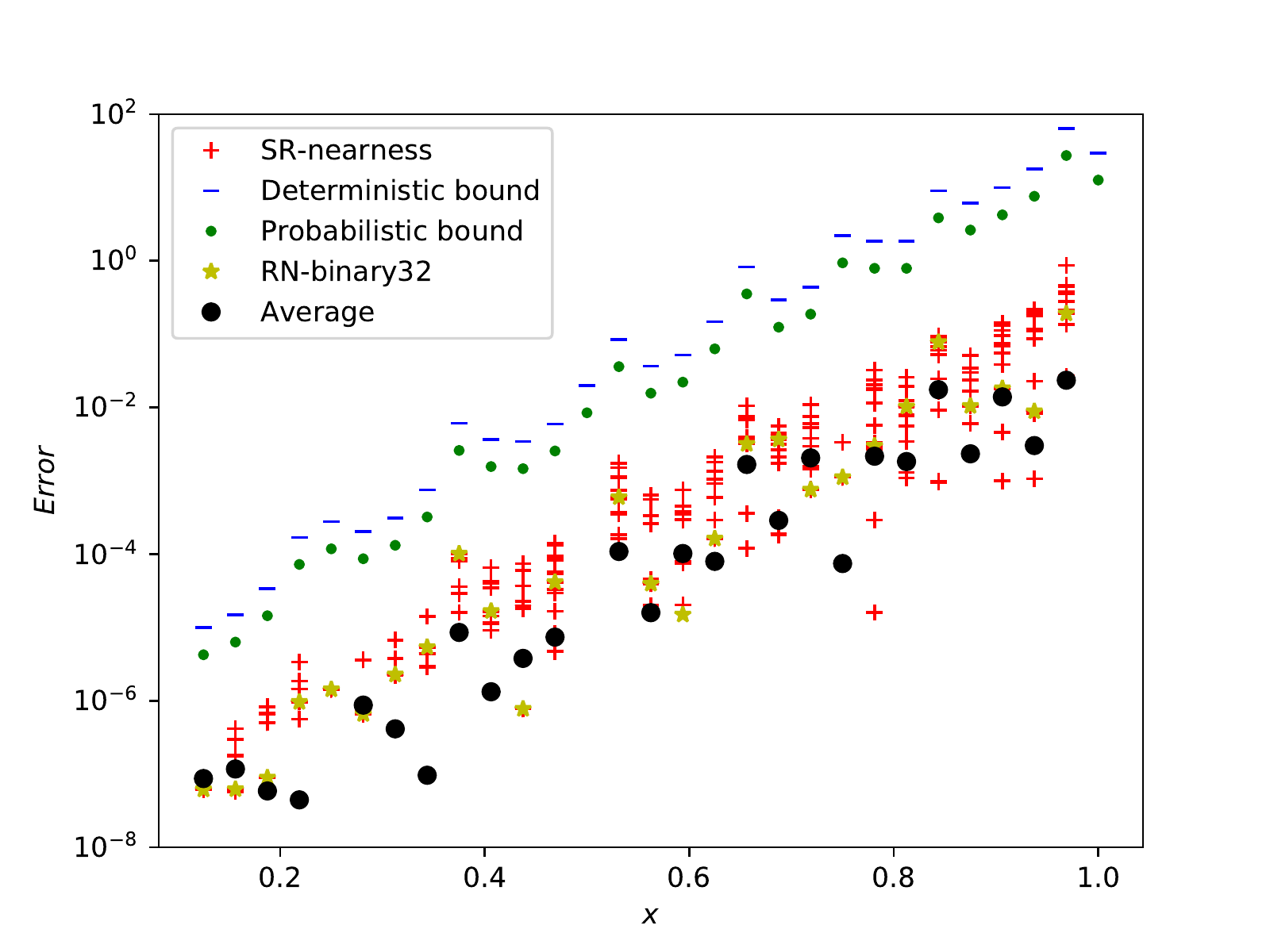}
    \caption{Probabilistic bound with $\lambda =0.5$ vs deterministic bound of the computed forward errors of Horner's rule for Chebyshev polynomial $T_{20}(x)$.}
    \label{fig:n=20}
\end{figure}

Chebyshev polynomial is ill-conditioned near $1$: in figure \ref{fig:n=20}, we evaluate $T_{20}(x)$ for $x\in[\frac{8}{64};1]$. As expected and due to catastrophic cancellations among the coefficients, the condition number increases from $10^0$ to $10^7$ which explains the increase of numerical error from $10^{-7}$ to $10^{-1}$. The probabilistic bound is closer to the forward error points than the deterministic bound even for a small $n=10$. The average of SR-nearness stays below RN-binary32 for almost all points.

In Figure \ref{fig:x=24/26}, the two previous bounds and the forward error are divided by the condition number $cond(P,x),$ and the evaluation in  $x=24/26\approx0.923$ is plotted for various polynomial degrees $N.$

\begin{figure}
    \centering
    \includegraphics[scale=0.55]{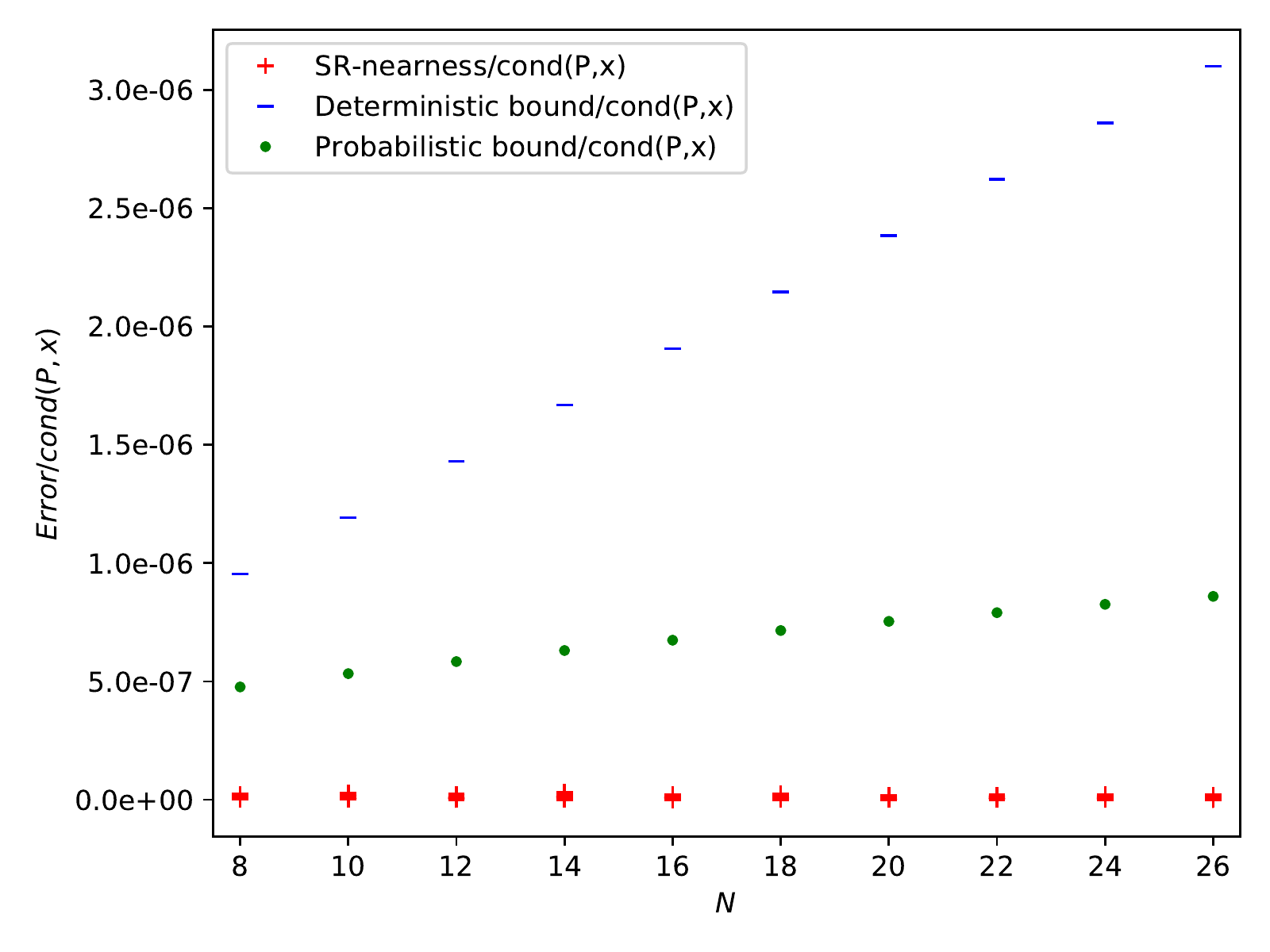}
    \caption{Forward errors/cond(P,x) of Horner's rule for Chebyshev polynomial $T_{N}(24/26)$.} 
    \label{fig:x=24/26}
\end{figure}

The comparison between the two bounds is fairly visible in figure \ref{fig:x=24/26}. By increasing  $N$, the deterministic bound draws away from the forward error faster than the probabilistic bound and for $N$ large, the gap becomes more interesting. SR-nearness points are between $10^{-8}$ and $10^{-10}$, versus $10^{-7}$ for the probabilistic bound and $10^{-6}$ for the deterministic bound. 

\section{Conclusion.}
Stochastic rounding has drawn a lot of attention in various domains~\cite{gupta,ode1,ode2,survey} due to its efficiency compared to the default rounding mode. The fact that SR-nearness satisfies mean independence (a weaker property than independence) leads to an expected value that coincides with the exact value. We have shown that the bias in SR-up-or-down can significantly reduce the precision of the computation, even on simple algorithms such as rectangular integration section~\ref{sec:ODE}, and that SR-nearness can remain unbiased and provide the full expected precision on them.
We also discussed that using SR-nearness leads to having a probabilistic bound in $O(\sqrt{n}u)$, compared to the $O(nu)$ deterministic bound for the inner product forward error. We have shown this property for Horner’s method using Azuma–Hoeffding inequality and martingale properties. As opposed to the study made for the inner product, the issue of this algorithm is that the martingale does not appear explicitly; nevertheless, a change of variable shows that it is present, allowing the use of concentration inequalities. As future work we will investigate more complex algorithms with non-explicit martingales.

\section*{Acknowledgment}
This research was supported by the French National Agency for Research (ANR) via the InterFLOP project (No. ANR-20-CE46-0009).

\bibliographystyle{bib_sobraep}

\bibliography{ref}

\begin{thebibliography}{10}
\newcommand{\enquote}[1]{``#1''}
\providecommand{\url}[1]{{\tt #1}}
\providecommand{\urlprefix}{URL: }
\expandafter\ifx\csname urlstyle\endcsname\relax
  \providecommand{\doi}[1]{doi:\discretionary{}{}{}#1}\else
  \providecommand{\doi}{doi:\discretionary{}{}{}\begingroup
  \urlstyle{rm}\Url}\fi
\providecommand{\eprint}[2][]{\url{#2}}

\bibitem{parker1997monte}
D.~S. Parker, {\em Monte Carlo Arithmetic: exploiting randomness in
  floating-point arithmetic\/}, University of California (Los Angeles).
  Computer Science Department, 1997.

\bibitem{Neumann1947NumericalIO}
J.~von Neumann, H.~H. Goldstine, \enquote{Numerical inverting of matrices of
  high order}, {\em Bulletin of the American Mathematical Society\/}, vol.~53,
  pp. 1021--1099, 1947.

\bibitem{graph-a}
\enquote{Graphcore Limited. 2021a IPU Programmer’s Guide. Version 2.0.0.}, .

\bibitem{graph-b}
\enquote{Graphcore Limited. 2021b Targeting the IPU from TensorFlow 1. Version
  2.0.0-dc2.}, .

\bibitem{graph-c}
\enquote{Graphcore Limited. 2021c AI-Float™- Mixed Precision Arithmetic for
  AI: A Hardware Perspective. Version latest: Aug 25, 2021.}, .

\bibitem{davies2018loihi}
M.~Davies, N.~Srinivasa, T.-H. Lin, G.~Chinya, Y.~Cao, S.~H. Choday, G.~Dimou,
  P.~Joshi, N.~Imam, S.~Jain, {\em et~al.\/}, \enquote{Loihi: A neuromorphic
  manycore processor with on-chip learning}, {\em Ieee Micro\/}, vol.~38,
  no.~1, pp. 82--99, 2018.

\bibitem{amd}
\enquote{{Loh GH. 2019 Stochastic rounding logic. Patent Status: Active.}}, .

\bibitem{nvidia}
J.~Alben, P.~Micikevicius, H.~Wu, M.~Siu, N.~Corporation, \enquote{Stochastic
  rounding of numerical values}, {\em Patent Status: Active\/}, 2019.

\bibitem{ibm1}
J.~D. Bradbury, S.~R. Carlough, B.~R. Prasky, E.~M. Schwarz,
  \enquote{Reproducible stochastic rounding for out of order processors}, US
  Patent 10,083,008, Sep.~25 2018.

\bibitem{ibm2}
J.~D. Bradbury, S.~R. Carlough, B.~R. Prasky, E.~M. Schwarz,
  \enquote{Stochastic rounding floating-point multiply instruction using
  entropy from a register}, US Patent 10,445,066, Oct.~15 2019.

\bibitem{com1}
G.~G. Henry, D.~R. Reed, \enquote{Processor with memory array operable as
  either cache memory or neural network unit memory}, US Patent 10,664,751,
  May~26 2020.

\bibitem{com2}
O.~A. Kanter, I.~Bar, \enquote{Apparatus and methods for hardware-efficient
  unbiased rounding}, US Patent 8,972,472, Mar.~3 2015.

\bibitem{com3}
S.~Lifsches, \enquote{In-memory stochastic rounder}, US Patent 10,803,141,
  Oct.~13 2020.

\bibitem{survey}
M.~Croci, M.~Fasi, N.~J. Higham, T.~Mary, M.~Mikaitis, \enquote{Stochastic
  Rounding: Implementation, Error Analysis, and Applications}, {\em Royal
  Society Open Science\/}, 2021.

\bibitem{verificarlo}
C.~Denis, P.~de~Oliveira~Castro, E.~Petit, \enquote{Verificarlo: Checking
  Floating Point Accuracy through Monte Carlo Arithmetic}, {\em in\/} {\em 23nd
  {IEEE} Symposium on Computer Arithmetic\/}, pp. 55--62, 2016,
  \doi{10.1109/ARITH.2016.31},
  \urlprefix\url{http://dx.doi.org/10.1109/ARITH.2016.31}.

\bibitem{cadna}
F.~J{\'e}z{\'e}quel, J.-M. Chesneaux, \enquote{CADNA: a library for estimating
  round-off error propagation}, {\em Computer Physics Communications\/}, vol.
  178, no.~12, pp. 933--955, 2008.

\bibitem{verrou}
F.~F{\'e}votte, B.~Lathuiliere, \enquote{VERROU: a CESTAC evaluation without
  recompilation}, {\em SCAN 2016\/}, p.~47, 2016.

\bibitem{norm}
\enquote{IEEE Standard for Floating-Point Arithmetic}, {\em IEEE Std 754-2019
  (Revision of IEEE 754-2008)\/}, pp. 1--84, 2019,
  \doi{10.1109/IEEESTD.2019.8766229}.

\bibitem{gupta}
S.~Gupta, A.~Agrawal, K.~Gopalakrishnan, P.~Narayanan, \enquote{Deep Learning
  with Limited Numerical Precision}, , 02 2015.

\bibitem{ml1}
M.~Höhfeld, S.~E. Fahlman, \enquote{Probabilistic rounding in neural network
  learning with limited precision}, {\em Neurocomputing\/}, vol.~4, no.~6, pp.
  291--299, 1992, \doi{https://doi.org/10.1016/0925-2312(92)90014-G}.

\bibitem{ml2}
M.~Hoehfeld, S.~Fahlman, \enquote{Learning with limited numerical precision
  using the cascade-correlation algorithm}, {\em IEEE Transactions on Neural
  Networks\/}, vol.~3, no.~4, pp. 602--611, 1992, \doi{10.1109/72.143374}.

\bibitem{ode1}
M.~Hopkins, M.~Mikaitis, D.~R. Lester, S.~Furber, \enquote{Stochastic rounding
  and reduced-precision fixed-point arithmetic for solving neural ordinary
  differential equations}, {\em Philosophical Transactions of the Royal Society
  A\/}, vol. 378, no. 2166, p. 20190052, 2020.

\bibitem{ode2}
M.~Fasi, M.~Mikaitis, \enquote{Algorithms for Stochastically Rounded Elementary
  Arithmetic Operations in IEEE 754 Floating-Point Arithmetic}, {\em IEEE
  Transactions on Emerging Topics in Computing\/}, vol.~9, no.~3, pp.
  1451--1466, 2021, \doi{10.1109/TETC.2021.3069165}.

\bibitem{theo19}
N.~J. Higham, T.~Mary, \enquote{A New Approach to Probabilistic Rounding Error
  Analysis}, {\em SIAM Journal on Scientific Computing\/}, vol.~41, no.~5, pp.
  A2815--A2835, 2019, \doi{10.1137/18M1226312}.

\bibitem{ilse}
I.~C.~F. Ipsen, H.~Zhou, \enquote{Probabilistic Error Analysis for Inner
  Products}, {\em SIAM Journal on Matrix Analysis and Applications\/}, vol.~41,
  no.~4, pp. 1726--1741, 2020, \doi{10.1137/19M1270434}.

\bibitem{theo21stocha}
N.~J.~H. M.~P.~Connolly, T.~Mary, \enquote{Stochastic rounding and its
  probabilistic backward error analysis}, {\em SIAM Journal on Scientific
  Computing\/}, vol. 43, No. 1, pp. A566--A585, 2021,
  \doi{https://doi.org/10.1137/20M1334796}.

\bibitem{verrou-tutorial}
F.~Févotte, B.~Lathuilière, P.~de~Oliveira~Castro, \enquote{{Etudier la
  qualité numérique d’un code avec Verrou}},
  \url{https://github.com/edf-hpc/verrou/releases/download/tutorials/tp-verrou.tgz},
  [Online; accessed 21-April-2021], 2018.

\bibitem{azum}
M.~Mitzenmacher, E.~Upfal, {\em Probability and Computing: Randomized
  Algorithms and Probabilistic Analysis\/}, Cambridge University Press, 2005,
  \doi{10.1017/CBO9780511813603}.

\bibitem{wilk}
J.~H. Wilkinson, {\em Rounding errors in algebraic processes\/}, Courier
  Corporation, 1994.

\bibitem{higham2002}
N.~J. Higham, {\em Accuracy and stability of numerical algorithms\/}, SIAM,
  2002.

\end{thebibliography}

\end{document}